\documentclass[10pt, reqno]{amsart}
\usepackage{color}
\usepackage[T1]{fontenc}
\usepackage{graphicx, enumerate}
\usepackage{amssymb, amsmath, amsthm}
\allowdisplaybreaks

\usepackage{hyperref}
\hypersetup{
    colorlinks=true, 
    linktoc=all,     
    linkcolor=blue,  
    citecolor=blue
}

\newtheorem{Theorem}{Theorem}[section]

\theoremstyle{definition}

\theoremstyle{remark}

\newtheorem*{Remark}{Remark}

\numberwithin{equation}{section}

\begin{document}

\title{Stanley--Elder--Fine theorems for colored partitions}

\author[H.~S.~Bal]{Hartosh Singh Bal
}
\address{The Caravan\\
Jhandewalan Extn., New Delhi 110001, India}
\email{hartoshbal@gmail.com}

\author[G.~Bhatnagar]{Gaurav Bhatnagar
}
\address{Ashoka University, Sonipat, Haryana, India}
\email{bhatnagarg@gmail.com}
\urladdr{http://www.gbhatnagar.com}

\date{\today}

\keywords{Integer partitions, Stanley's theorem, Elder's theorem, colored partitions, prefabs, partitions with $k$ colors of $k$}
\subjclass[2010]{Primary: 11P81; Secondary: 05A17}

\begin{abstract}
We give a new proof of a partition theorem popularly known as Elder's theorem, but which is also credited to Stanley and Fine. We extend the theorem to the context of colored partitions (or prefabs). 
More specifically, we give analogous results for $b$-colored partitions, where each part occurs in $b$ colors; for $b$-colored partitions with odd parts (or distinct parts); for partitions where the part $k$ comes in $k$ colors; and, overpartitions.
\end{abstract}

\maketitle

\section{Introduction}
The purpose of this paper is to extend a charming theorem in the theory of partitions which appeared in Stanley \cite[Ch~1, Ex.~ 80]{Stanley2012}, but is usually attributed to Elder, and  more recently, has been found in the work of Fine; see Gilbert \cite{Gilbert2015} for a comprehensive history. Our extensions appears in the context of colored partitions (or prefabs). As a consequence, we give analogous results for partitions where the part $k$ occurs in $b$ colors, and partitions where the part $k$ appears in $k$ colors. We also consider 
colored partitions with odd or distinct parts and overpartitions.  
In the process, we extend results of  Andrews and Merca~\cite{AM2020} and Gilbert~\cite{Gilbert2015}. 

We recall some of the terminology from the theory of integer partitions.
A partition of $n$ is a way of writing $n$ as an unordered sum of numbers. 
It is represented as a sequence of non-increasing, positive integers $$\lambda: \lambda_1 \ge \lambda_2\ge \lambda_3 \ge \dots ,$$
with 
$$n=\lambda_1+\lambda_2+\lambda_3+\cdots.$$
The symbol
$\lambda \vdash n$  
is used to say that $\lambda$ is a partition of $n$; we also say $\lambda$ has weight $n$. 
If $\lambda \vdash n$, we have
\begin{equation*}\label{freq1}
n=\sum_k k f_k
\end{equation*}
where the non-negative integers $f_k=f_k(\lambda)$ denote the frequency of $k$ in $\lambda$, that is, the number of times $k$ comes in $\lambda$. 
For example, the partition $4+3+3+2+1+1+1+1$ has frequencies: 
$f_1=4, f_2=1, f_3=2, f_4=1$. 

One of the quantities in the Stanley--Elder--Fine theorem is
$$F_k(n) := \sum_{\lambda \vdash n } f_k(\lambda),$$
the total number of $k$'s appearing in all the partitions of $n$. 
The other is
$$G_k(n) = \sum_{\lambda \vdash n } g_k(\lambda),$$
where
$$g_k(\lambda):=  \text{the number of parts which appear at-least $k$ times in $\lambda$.}
$$

The Stanley--Elder--Fine theorem says that for all $n$,
\begin{equation}\label{SEF}
F_k(n) = G_k(n),
\end{equation}
for $k=1, 2, \dots, n$. 

There are several proofs of this result, many of a combinatorial nature (see \cite{Gilbert2015} for references). Here is another, rather simple, combinatorial proof of \eqref{SEF}. 
We first show
\begin{equation}\label{Fk-partitions}
F_k(n)=p(n-k)+p(n-2k)+p(n-3k)+\cdots,
\end{equation}
for $k=1, 2, \dots $. (We take $p(m)=0$ for $m<0$.) 
Observe that for $k=1, 2, \dots, n$,
\begin{equation*}
F_k(n)=p(n-k)+F_k(n-k),
\end{equation*}
because adding $k$ to each partition of $n-k$ yields a partition of $n$, and vice-versa, on deletion of $k$ from any partition containing $k$ as a part, we obtain a partition of $n-k$.  
This gives \eqref{Fk-partitions} by iteration.
 
Next, consider $G_k(n)$.
%
If we add $1+1+\cdots+1$ ($k$-times) to any partition of $n-k$, we obtain a partition of 
$n$ where $1$ appears as a part at-least $k$ times; 
if we add $2+2+\cdots+2$ ($k$-times) to any partition of $n-2k$, we obtain a partition of $n$ where $2$ 
appears as a part at-least $k$ times; and so on.  The process is reversible. Thus 
\begin{equation}\label{G-k}
G_k(n)=p(n-k)+p(n-2k)+p(n-3k)+\cdots
\end{equation}
for  $k=1, 2, 3, \dots$; and $F_k(n)=G_k(n)$ for all $k=1, 2, \dots, n$. 

The ``counting by rows = counting by columns'' quality of the Stanley--Elder--Fine theorem is reflected in this proof. 

The objective of this paper is to extend this proof  to colored partitions generated by  
$$\prod_{k=1}^\infty\frac{1}{\left(1-q^k\right)^{b_k}},$$
where $b_k$ is a sequence of non-negative numbers. 
These are called prefabs by Wilf \cite[\S 3.14]{wilf2006}, but we prefer the imagery of partitions with colored parts. Each part $k$ comes in $b_k$ colors. They reduce to ordinary partitions when $b_k=1$ for all $k$. We are able to extend \eqref{SEF} to the cases $b_k=b$ (a positive number) and when $b_k=k$ for all $k$. In addition, we consider partition objects from generating functions that are products of such products; in particular, we consider overpartitions. 

This paper is organized as follows. In \S\ref{sec:frequency}, we prove an analogue of
\eqref{Fk-partitions} for colored partitions. In \S\ref{sec:colored} we consider $b$-colored partitions, where each part has $b$ colors. Next, in \S\ref{sec:odd}, we consider $b$-colored partitions with odd or distinct parts.  In \S\ref{sec:plane} we consider partitions where the part $k$ comes in $k$ colors. The number of such partitions is the same as the number of plane partitions of $n$. Next, in \S\ref{sec:overpartitions} we consider overpartitions. We conclude in \S\ref{sec:credits} by giving credit where credit is due. 


\section{The frequency function for colored partitions}\label{sec:frequency}
The objective of this section is to obtain a key relation for the frequency function for colored partitions.  We use notation from \cite{BB2021a} to represent colored partitions.  

Let  $u_k$ represent a set containing $b_k$ copies of $k$. The elements of $u_k$ are represented as $k_1, k_2, \dots$. The elements of $u_k$ can be regarded as $k$ with different colors. 
Consider two sets $u_j$ and $u_k$, containing, respectively, $b_j$ and $b_k$ colors.
 We now use the symbol
$u_j+u_k$ to denote the set of partitions
$$u_j+u_k:=\{j_a+k_b: j_a\in u_j, k_b\in u_k\}.$$
Here $2u_j$ represents $u_j+u_j$. 
This definition is extended by induction to finite sums $\sum a_i u_i$ where $a_i\geq 0$. 
By a colored partition of weight $n$, we mean an element $\pi$ where
$$\pi\in \sum_i a_i u_i , $$
with $|\pi|=\sum_{i=1}^n i a_i = n.$

For a partition $\pi$, let $f_{k_c}(\pi)$ be the number of parts equal to $k_c$ in $\pi$, i.e., the frequency of 
$k_c$ in $\pi$. Then the frequency of $k$ in $\pi$ is
$$f_k(\pi) = \sum_{k_c\in x_k} f_{k_c}(\pi).$$
We denote the sum of the frequencies of all partitions of size $n$  by $F_k(n)$.

%
%

For a partition $\pi$, let $f_{k_c}(\pi)$ be the number of parts equal to $k_c$ in $\pi$, i.e., the frequency of 
$k_c$ in $\pi$. Then the frequency of $k$ in $\pi$ is
$$f_k(\pi) = \sum_{k_c\in x_k} f_{k_c}(\pi).$$
We denote the sum of the frequencies of all partitions of size $n$  by $F_k(n)$.

\begin{Theorem}\label{prefab-freq-matrix}
Let $h(n)$ represent the number of colored partitions of size $n$ where $k$ comes in $b_k$ colors. Let $F_k(n)$ be the frequency of $k$ in all the partitions of $n$. Then we have the recurrence relation
\begin{equation}
F_k(n) = b_k\left(h(n-k)+h(n-2k)+h(n-3k)+\cdots \right)\label{recurrence-2}
\end{equation}
\end{Theorem}
\begin{Remark} Here $b_k$ is a sequence of non-negative integers. A more general version of this theorem, where the $b_k$ are complex numbers, is proved in \cite{BB2021a}.
\end{Remark}
\begin{proof} We prove
\begin{equation}
F_k(n-k)+b_k h(n-k). \label{recurrence-partitions}
\end{equation}
For $X=\sum_j c_j u_j$, let $|X|$ denote the number of partitions in $X$, and let $F_k(X)$ denote the number of $k$'s in $X$. 
Consider a set of partitions of weight $n$ represented by $X=c_k u_k+Y$, where $Y=\sum_{j\neq k} c_j x_j$. Here $c_k>0$.  Now the number of partitions in $X$ are 
$$\binom{c_k+|u_k|-1}{|u_k|-1} |Y|=\binom{c_k+b_k-1}{b_k-1} |Y|.$$
We relate $X$ with partitions obtained by deleting one $k$. 

 The contribution to $F_k(n)$ from $X$ is
 $$F_k(X)=c_k\binom{c_k+b_k-1}{b_k-1} |Y|.$$
Using the elementary identity 
$$(n+1)\binom{n+k}{k-1} = k\binom{n+k-1}{k-1}+n\binom{n+k-1}{k-1}$$
we find that
$$F_k(X) =  (c_k-1)\binom{c_k-1+b_k-1}{b_k-1} |Y|+ b_k \binom{c_k-1+b_k-1}{b_k-1} |Y|.$$

Now the first of the two terms on the right is $F_k(X-u_k)$, the number of $k$'s in $X-u_k$. Thus on summing over all $X$ that contains $u_k$, we obtain $F_k(n-k)$. 
In the second term, the quantity 
$$\binom{c_k-1+b_k-1}{b_k-1} |Y|$$
is the number of partitions in $X-u_k$. The weight of each partition is $n-k$. Summing over all the partitions $X$ of weight $n$ which contain $u_k$, we obtain $h(n-k)$. 
This shows \eqref{recurrence-partitions}.
\end{proof}

\section{Partitions with the same number of colors for each part}\label{sec:colored}
We consider $b$-colored partitions where each part $k$ comes in $b$ colors, where $b$ is a positive integer. A Stanley--Elder--Fine theorem is as easy to obtain in this context as the 
$b=1$ case.  Let $h(n)$ represent the number of $b$-colored partitions of $n$. 
Let $k_1$, $k_2$, $\dots$, $k_b$ represent the colored parts. 
Let $F_k(n)$ be the frequency of $k$ in all the partitions of $n$. 
From \eqref{recurrence-2}, we have 
\begin{equation}\label{SEF-partitions}
F_k(n)=b h(n-k)+b h(n-2k)+b h(n-3k)+\cdots.
\end{equation}
Let $\pi$ be a partition. As before, let
$$g_{k}(\pi):=  \text{the number of parts which appear at-least $k$ times in $\pi$,}
$$
and 
$$G_k(n) = \sum_{|\pi| = n } g_k(\pi).$$
Thus $G_k(n)$ is the number of times a part appears at-least $k$ times in a partition, summed over all the partitions of $n$. Then we have
\begin{Theorem}\label{th:SEF} Let $F_k(n)$ and $G_k(n)$ be as above. Then, 
for all $n=1, 2, \dots$,
$$F_k(n) = G_k(n),$$
for $k=1, 2, \dots, n$. 
\end{Theorem}
\begin{proof} The proof is virtually the same as the $b=1$ case. 
Note that if we add a $r$ (of any color) to any partition of $n-r$, we obtain a partition of $n$ which has at-least 
one $r$ as a part. We can add any one of the $b$ $r'$s. This can be done for each 
$r=1, 2, 3, \dots$. Thus
$$G_1(n) = bh(n-1)+bh(n-2)+\cdots + bh(0),$$
since every part from every partition of $n$ which is repeated at-least once will be accounted
 for (uniquely) in this way. 

In general, we see that if we add $1_c+1_c+\cdots+1_c$ ($k$-times) to any partition of $n-k$, we obtain a partition of 
$n$ where $1_c$ appears as a part at-least $k$ times; 
if we add $2_c+2_c+\cdots+2_c$ ($k$-times) to any partition of $n-2k$, we obtain a partition of $n$ where $2_c$ 
appears as a part at-least $k$ times; and so on.  Thus 
$$G_k(n)=bh(n-k)+bh(n-2k)+bh(n-3k)+\cdots$$
for  $k=1, 2, 3, \dots$.
In view of \eqref{SEF-partitions}, $G_k(n)$ and $F_k(n)$ are equal. 
\end{proof}

Next, we consider the quantity $H_k(n)$, defined as the sum of parts divisible by $k$, counted without multiplicity, in all the $b$-colored partitions of $n$. 
As an example, consider an ordinary partition (that is, $b=1$) represented by $2u_3+4u_6$
or $3+3+6+6+6+6$. This contributes $3+6=9$ to the sum.
On the other hand, a 2-colored partition $3_1+3_2+6_1+6_1+6_1+6_2\in 2u_3+4u_6$ contributes $3+3+6+6=18$ to $H_3(30)$.
\begin{Theorem}
Let $F_k(n)$ and $H_k(n)$ be as above. Then for all $n$, we have
$$kF_k(n)=H_k(n)-H_k(n-k).$$
\end{Theorem}
\begin{Remark} When the number of colors $b=1$, i.e., in the case of ordinary partitions, this result reduces to a result of Andrews and Merca \cite{AM2020}. 
\end{Remark}
\begin{proof}
We first show
\begin{equation}\label{H-recursion}
H_k(n)=bkh(n-k)+2bkh(n-2k)+3bkh(n-3k)+\cdots.
\end{equation}
The argument is similar to the one for $G_k(n)$. If we add $(rk)_c$ (the part $rk$ in color $c$) to any partition of $n-rk$, we obtain a partition with $(rk)_c$ as a part. This contributes $rk$ to $H_k(n)$. Conversely, if we delete $(rk)_c$ in a partition of $n$ where it comes as a part, we obtain a partition of $n-rk$. 
This shows \eqref{H-recursion}. 

Now it is clear that 
$H_k(n)-H_k(n-k)$ 
equals $kF_k(n)$ by \eqref{SEF-partitions}. 
\end{proof}

\section{Partitions with odd and distinct parts}\label{sec:odd}
We consider  $b$-colored  partitions with all parts odd (which all come in $b$-colors).  Let $h(n)$ now denote the number of $b$-colored partitions with only odd parts. An easy extension of Euler's ODD=DISTINCT theorem (see \cite[eq.~(2.1)]{AE2004}) says that $h(n)$ is also the number of $b$-colored partitions with distinct parts.
Here the $h(n)$ are generated by
$$\prod_{k=1}^\infty {\big(1+q^k\big)^b}=\prod_{k=1}^\infty \frac{1}{\big(1-q^{2k-1}\big)^b}.$$

Let $F^o_k(n)$ denote the corresponding frequency function. Then we have
\begin{subequations}
\begin{align}
 F^o_k(n) &= \label{freq-odd-h-a}
\begin{cases} 
 F^o(n-k)+bh(n-k), 
& \text{ if  $k$ is odd;} \\
 0, & \text{ if  $k$ is even.} 
 \end{cases} \\
&= \label{freq-odd-h-b}
 \begin{cases} 
 h(n-k)+h(n-2k) + h(n-3k) + \cdots, 
& \text{ if  $k$ is odd;} \\
 0, & \text{ if  $k$ is even.} 
 \end{cases}
\end{align}
\end{subequations}
Let $G_k^o(n)$ be the number of times a part appears at-least $k$ times in a partition, summed over all the $b$-colored partitions of $n$ with odd parts.

\begin{Theorem} Let $F_k^o(n)$ and $G_k^o(n)$ be as above and let $k$ be an odd number. Then, 
for all $n=1, 2, \dots$,
$$F_k^o(n) = G_k^o(n)+G_k^o(n-k),$$
for $k=1, 3, 5, \dots $. 
\end{Theorem}
\begin{Remark} In the case of ordinary partitions, where $b=1$, this theorem reduces to an observation of 
Gilbert \cite[Th.\ 8]{Gilbert2015}.
\end{Remark}
\begin{proof}
If we add $1_c+1_c+\cdots+1_c$ ($k$-times) to any partition of $n-k$, we obtain a partition of 
$n$ where $1_c$ appears as a part at-least $k$ times; 
if we add $3_c+3_c+\cdots+3_c$ ($k$-times) to any partition of $n-3k$, we obtain a partition of $n$ where $3_c$ 
appears as a part at-least $k$ times; and so on.  Thus 
\begin{equation}\label{Gk-odd-h}
G_k^o(n)=bh(n-k)+bh(n-3k)+bh(n-5k)+\cdots
\end{equation}
for  $k=1, 2, 3, \dots$. Note that this applies even if $k$ is not an odd number. 

When $k$ is odd, we see from  \eqref{freq-odd-h-b} that
$$G_k^o+G_k^o(n-k) = F_k^o(n),$$
This proves the theorem. 
\end{proof}
Next let $F_k^d(n)$ denote the frequency of $k$ in $b$-colored partitions with all parts distinct. The differently colored parts of the same weight are considered distinct. For example, $3_1+3_2+5_2$ is  considered a $2$-partition of $11$ with distinct parts.
This partition contributes $2$ to $F_3^d(11)$. 
It is easy to see that 
\begin{subequations}
\begin{align}
F_k^d (n) &= b h(n-k)-F_k^d(n-k) \label{F-distinct-h-a} \\
& = b h(n-k)-bh(n-2k)+bh(n-3k)-bh(n-4k)+\cdots .\label{F-distinct-h-b}
\end{align}
\end{subequations}
To obtain \eqref{F-distinct-h-a}, observe that a $b$-colored partition of $n$ (with distinct parts) which contains $k_c$ as a part is obtained 
by adding $k_c$ to any partition of $n-k$ which does not have $k_c$ as a part. So the number of 
distinct partitions of $n$ containing $k_c$ is  $h(n-k)-F_{k_c}(n)$, where $F_{k_c}(n)$ is number of $k_c$'s in the distinct partitions of $n-k$ (and also the number of distinct partitions containing $k_c$ as a part). Now summing over all colors, we obtain \eqref{F-distinct-h-a}.

Since the number $h(n)$ of  $b$-colored partitions with odd parts and distinct parts are the same, equation \eqref{F-distinct-h-b}, 
along with \eqref{Gk-odd-h} immediately yields the following theorem. 
\begin{Theorem}\label{th:d-o} Let $F_k^d(n)$ and $G_k^o(n)$ be as defined above. Then
$$F_k^d(n) = G_k^o(n)-G_k^o(n-k),$$
for $k=1, 2, 3, \dots$.
\end{Theorem}
\begin{Remark} When $b=1$, Theorem~\ref{th:d-o} reduces to Gilbert \cite[Th.\ 9]{Gilbert2015}.
\end{Remark}

\section{Partitions with $k$ copies of $k$}\label{sec:plane}
Next we consider a special case of colored partitions generated by the product
$$\prod_{k=1}^\infty\frac{1}{\left(1-q^k\right)^{k}}.$$
We reuse the notation $h(n)$ to denote the number of such partitions of $n$.
The notations for $F_k(n)$ and $G_k(n)$ are also reused. 
\begin{Theorem} Consider the set of partitions of $n$ where each part $k$ comes in $k$ colors. Let $F_k(n)$ 
be the number of $k$'s (of any color) appearing in all such partitions of $n$. Let
$G_k(n)$ denote the number of parts that appear at-least $k$ times in  such a partition, summed over all such partitions.  Then, 
for all $n=1, 2, \dots$,
$$F_k(n) = k\big( G_k(n)-G_k(n-k)\big),$$
for $k=1, 2, \dots, n$. 
\end{Theorem}

\begin{proof}
Observe that
$$G_k(n)=1\cdot h(n-k)+2\cdot h(n-2k)+ 3\cdot h(n-3k)+\cdots$$
for  $k=1, 2, 3, \dots$.

Thus we have 
$$G_k(n)-G_k(n-k) = h(n-k)+h(n-2k)+\cdots,$$
and, by Theorem~\ref{prefab-freq-matrix},
\begin{align*}
F_k(n)&=k \big( h(n-k)+ h(n-2k)+ h(n-3k)+\cdots\big)\\
&= k\big(G_k(n)-G_k(n-k)\big).
\end{align*}
\end{proof}

\section{Overpartitions}\label{sec:overpartitions}
An overpartition of $n$ is a non-increasing sequence of natural numbers whose sum is $n$, where the first occurrence of a  number can be overlined. We denote the number of overpartitions of $n$ by
$\overline{p}(n)$. These can be represented as partitions in two symbols $u$ and $v$. The partitions represented by $u$ are ordinary partitions and the partitions represented by $v$ are distinct. In a sense (to be explained shortly), overpartitions are convolutions of these two types of partitions. One of our theorems in this section shows how one can simply put together the respective results for two partition functions to obtain a new theorem for their convolution. In addition, we give another extension of Stanley's theorem, which also follows by manipulating recurrence relations. Its proof is more intricate than what we have encountered so far.

The overpartitions of $n$ can be formed by adding  partitions of $k$ in $u$ with a distinct partition of $n-k$ in $v$. For example, here is a way to list the overpartitions of $4$. First we list ordinary partitions up to $4$ and then add them with partitions with distinct parts written in reverse order. In Table~\ref{table1} we have listed the partitions of $4$ and partitions with distinct parts in reverse order in $v$.
\begin{table}[h]
$$
\begin{array}{| l | c | | c| | c | }
\hline
m & \parbox{3cm}{Partitions of $m$} & \parbox{3cm}{Partitions of $n-m$ into distinct parts}& \parbox{3cm}{Partitions of $n-m$ into odd parts}\\
\hline 
0 & \text{-} &  v_1+v_3 , v_4 & 4w_1, w_1+w_3\\
\hline
1 & u_1 &  v_1+v_2, v_3 & 3w_1, w_3 \\
\hline
2 & 2u_1, u_2 & v_2 &2w_1\\
\hline
3 & 3u_1,u_1+ u_2, u_3 &  v_1 & w_1\\
\hline
4 & 4u_1,2u_1+ u_2, u_1+u_3, 2u_2, u_4  &\text{-} &\text{-}  \\
\hline
\end{array}
$$
\caption{Generating overpartitions and odd-overlined partitions}\label{table1}
\end{table}
For example, $2u_1+v_2$ represents $\overline{2}+1+1$ and $u_2+v_2$ represents $\overline{2}+2$.
Evidently, 
$$\sum_m p(m)p(n-m \;|\text{ distinct parts}) = \overline{p}(n),$$
a convolution of two sums; thus
the 
generating function of overpartitions is the product of the respective generating functions: 
$$\overline{Q}(q)=
\sum_{n\geq 0} \overline{p}(n)q^n  
=
\prod_{k=1}^\infty \frac{1+q^k}{1-q^k}.$$
It is in this sense we describe overpartitions as convolutions of ordinary partitions with partitions of distinct parts.

Since the number of partitions into distinct parts equals the number of odd partitions, it is natural to consider the partitions formed by adding an ordinary partition of $m$ in $u$ with a partitions of $n-m$ with odd parts in $w$ (see Table~\ref{table1}). These are equinumerous to overpartitions. We call them {\em odd-overlined} partitions. 

Consider colored overlined partitions (of both kinds) generated by the generating functions
$$\sum_{n=0}^\infty h(n)q^n=\prod_{k=1}^\infty \frac{{\big(1+q^k\big)^s}}{\big(1-q^k\big)^r}=
\prod_{k=1}^\infty \frac{1}{\big(1-q^k\big)^r \big(1-q^{2k-1}\big)^s}.$$
Here the ordinary parts are $r$-colored and the distinct/odd parts are colored in $s$ colors. 

%
We mix and match the notations of \S\ref{sec:colored} and \S\ref{sec:odd}. So, for example, $h(n)$ will refer to the number of overpartitions (respectively, odd-overlined partitions), $F_k(n)$ refers to the frequency of $k$ in appearing in ordinary partition (in $r$ colors) contained in the overline partition, and $F^d_k(n)$ and $F^o_k(n)$ are the frequencies of $k$ of the overlined parts which come in $s$ colors.
Similarly, let $G_k(n)$, $G_k^o(n)$ be defined as earlier. Then we have:
\begin{Theorem}  
Let $F_k(n)$, $F_k^o(n)$, $F^d_k(n)$, $G_k(n)$, and $G_k^o(n)$ be as defined above, in the context of colored overpartitions/odd-overlined partitions. 
Then,  for all $n=1, 2, \dots$,
\begin{align*}
F_k(n) & = G_k(n), \text{ for $k=1, 2, 3, \dots $;}  \\
F_k^d(n) &= G_k^o(n)-G_k^o(n-k), \text{ for $k=1, 2, 3, \dots $;}\\ 
F_k^o(n) &= G_k^o(n)+G_k^o(n-k), \text{ for $k=1, 3, 5, \dots $.}\\ 
\end{align*}
\end{Theorem}
\begin{Remark}
One can get analogous theorems for partitions where each part $k$ comes in $k+b$ colors, by combining the considerations of \S\ref{sec:colored} with \S\ref{sec:plane}.
\end{Remark}

Before concluding, we give one more theorem concerning overpartitions, which is of a different nature than those studied above. Consider overpartitions with all parts of a single color generated by $\overline{Q}(q)$. For this theorem, we prefer the imagery of overpartitions where the part in $v$ is overlined. 

Let $F_k(n)$ and $F_d(k)$ be as above. Let 
$$\overline{F}_k(n)=F_k(n)+F_k^d(n).$$
So $\overline{F}_1(n)$ is the number of overpartitions of $n$ with $1$ or $\overline{1}$ as parts. 
Let $\overline{G}_k(n)$ be the number of overpartitions where a part is repeated at-least $k$ times. For example, the overpartition $3+2+2+\overline{2}+1+\overline{1}$ has the part $2$ repeated three times and contributes $1$ to $\overline{G}_k(11)$ for $k=1, 2, 3$. 

Thus $\overline{G}_1(n)$ is the total number of parts $m$ in overpartitions of $n$ such that $m$ or $\overline{m}$ appear at-least once in a partition of $n$, and $\overline{G}_3(n)$ is the number of parts repeated thrice or more. 
An extension of Stanley's theorem (that it, the $b=1$ and $k=1$ case of Theorem~\ref{th:SEF}) to
overpartitions says that  the number of overpartitions of $n$ with $1$ or $\overline{1}$ as a part
is the difference of these two quantities. 
\begin{Theorem} Let $\overline{F}_k(n)$ and $\overline{G}_k(n)$ be as above. Then, for
$n=1, 2, \dots$,
$$\overline{F}_1(n)= \overline{G}_1(n)-\overline{G}_3(n).$$
\end{Theorem}
\begin{proof}
It is easy to see (by the arguments used to obtain \eqref{recurrence-2} and \eqref{F-distinct-h-b}) that
\begin{align*}
F_1(n) & = \overline{p}(n-1)+\overline{p}(n-2)+\overline{p}(n-2)+\overline{p}(n-3)+\cdots,\\
\intertext{and}
F^d_1(n)& = \overline{p}(n-1)-\overline{p}(n-2)+\overline{p}(n-2)-\overline{p}(n-3)+\cdots,
\end{align*}
so
\begin{equation}\label{F1-overpartitions}
\overline{F}_1(n)= 2\big( \overline{p}(n-1)+\overline{p}(n-3)+\overline{p}(n-5)+\cdots\big) .
\end{equation}

To obtain an expression for the right hand side, we need the following ancillary counting functions.
Let $\overline{p}_m(n)$ be the number of overlined partitions which have $m$ as a part. Note that
\begin{equation}\label{overline1}
\overline{p}_m(n) = \overline{p}(n-m).
\end{equation}
We also need the following functions:
\begin{align*}
\overline{O}_m(n) &:=\parbox{3.25in}{ the number of overpartitions where the part $m$ or $\overline{m}$ appears at-least {\em once};}\\
{O}_m(n) &:=\parbox{3.25in}{ the number of overpartitions where the part $m$ appears at-least once;}\\
{O}_{\overline{m}}(n) &:=\parbox{3.25in}{ the number of overpartitions where the part $m$ does not appear and 
$\overline{m}$ appears;}\\
\overline{T}_m(n) &:=\parbox{3.25in}{ the number of overpartitions where the part $m$ or $\overline{m}$ appears at-least {\em thrice}.}
\end{align*}
Evidently $$\overline{O}_m(n)={O}_m(n)+{O}_{\overline{m}}(n) ;$$
\begin{align*}
\sum_{m=1}^n \overline{O}_m(n) &= \overline{G}_1(n);\\
\intertext{and,}
\sum_{m=1}^n \overline{T}_m(n) &= \overline{G}_3(n).
\end{align*}
To prove the theorem, we find expressions for $\overline{O}_m(n)$ and $\overline{T}_m(n)$
in terms of  $\overline{p}(n)$.

Note that 
\begin{align*}
O_1(n)&=p(n-1) \\ 
O_{\overline{1}}(n) & =\overline{p}(n-1)-\overline{p}_1(n-1)-O_{\overline{1}}(n-1)\\ 
&= \overline{p}(n-1)-2\overline{p}(n-2)+2\overline{p}(n-3)-\overline{p}(n-4)+\cdots .
\end{align*}
The first of these follows because we obtain an overpartition of $n$ with $1$ as a part by adding a $1$ to each overpartition of $n-1$. For $O_{\overline{1}}(n)$ we note that we can add a $\overline{1}$ to each overpartition of $n-1$, which has neither $1$ nor $\overline{1}$ as a part. Finally, the last line follows from \eqref{overline1} and iteration.

From the above, we find that
\begin{equation*}
\overline{O}_1((n)=
2\left(\overline{p}(n-1)-\overline{p}(n-2)+\overline{p}(n-3)-\overline{p}(n-4)+\cdots\right).
\end{equation*}
Similarly, we have, for $m=1, 2, \dots, n$
\begin{equation}\label{O-m}
\overline{O}_m((n)=
2\left(\overline{p}(n-m)-\overline{p}(n-2m)+\overline{p}(n-3m)-\overline{p}(n-4m)+\cdots\right).
\end{equation}

Next, we note that
\begin{align*}
\overline{T}_{1}(n) & =\overline{p}(n-3)+O_{\overline{1}}(n-2)\\
&=
2\left(\overline{p}(n-3)-\overline{p}(n-4)+\overline{p}(n-5)-\overline{p}(n-7)+\cdots\right)  
\end{align*}
The first of these is true because any overpartition where $1$ comes at-least three times is obtained 
by
adding $1+1+1$ to an overpartition of $n-3$ or by adding $1+1$ to an overpartition of $n-2$ which does not have a $1$ but has an $\overline{1}$. The second follows by using the formula for 
$O_{\overline{1}}(n)$ computed above.

Similarly, for $m=1, 2, \dots, n$, 
\begin{equation}\label{T-m}
\overline{T}_m(n)=
2\left(\overline{p}(n-3m)-\overline{p}(n-4m)+\overline{p}(n-5m)-\overline{p}(n-6m)+\cdots\right)
\end{equation}

Finally, we have
\begin{align*}
\overline{G}_1(n)-\overline{G}_3(n) &=
\sum_{m=1}^n \overline{O}_m(n)-\overline{T}_m(n) \\
&= \sum_{m=1}^n 2\big(\overline{p}(n-m)-\overline{p}(n-2m)\big) 
\text{ (using \eqref{O-m} and \eqref{T-m})}\\
&=2\big( \overline{p}(n-1)+\overline{p}(n-3)+\overline{p}(n-5)+\cdots\big)\\
&=\overline{F}_1(n),
\end{align*}
using \eqref{F1-overpartitions}. This completes the proof.
\end{proof}

\section{Closing credits}\label{sec:credits}
The key idea in our extensions of the Stanley--Elder--Fine theorem is the recurrence
\eqref{recurrence-partitions} from which \eqref{recurrence-2} follows.  From here, it is easy to manipulate the expression for $G_k(n)$ corresponding to the choice of $b_k$. Even so, the corresponding theorems  of Andrews and Merca \cite{AM2020} and Gilbert~\cite{Gilbert2015} have  motivated the form of our theorems. In particular, the definition of $H_k(n)$ for ordinary partitions in \cite{AM2020} was very useful. 

As we saw when considering overpartitions, we can mix and match and find Stanley--Elder--Fine theorems for partitions that can be constructed by the convolution of two different kinds of partitions. In addition to overpartitions, many such partitions have appeared in the literature, and this technique can be used to give such theorems of them.

We mention some related work. Banerjee and Dastidar~\cite{BD2020} have given a proof by combinatorial means too, but it is much more intricate than the one given here. Their colored partitions are different from ours; they are closer in spirit to the work in \S~\ref{sec:overpartitions}.
Dastikar and Sen Gupta \cite{DS2013} has given an extension of Stanley's theorem which comes from summing \eqref{G-k} for $k=1, 2, \dots, k$ and noting that the sum equals $F_1(n)$. Their results can be immediately extended (as in this paper) to the context of colored partitions.  Andrews and Deutsch \cite{AD2016} have a different generalization of Elder's theorem. Their starting point and key argument is not far from ours, but they have taken a different path to generalization; see
also \cite{AM2017}. 

The relations \eqref{recurrence-partitions} and \eqref{recurrence-2} have number-theoretic consequences. These are studied by the authors in \cite{BB2021a}.


\begin{thebibliography}{10}

\bibitem{AM2017}
A.~M. Alanazi and A.~O. Munagi.
\newblock On partition configurations of {A}ndrews-{D}eutsch.
\newblock {\em Integers}, 17:Paper No. A7, 12, 2017.

\bibitem{AM2020}
G.~Andrews and M.~Merca.
\newblock A new generalization of {S}tanley's theorem.
\newblock {\em Math. Student}, 89(1-2):175--180, 2020.

\bibitem{AD2016}
G.~E. Andrews and E.~Deutsch.
\newblock A note on a method of {E}rd{\H o}s and the {S}tanley-{E}lder
  theorems.
\newblock {\em Integers}, 16:Paper No. A24, 5, 2016.

\bibitem{AE2004}
G.~E. Andrews and K.~Eriksson.
\newblock {\em Integer partitions}.
\newblock Cambridge University Press, Cambridge, 2004.

\bibitem{BB2021a}
H.~S. Bal and G.~Bhatnagar.
\newblock The {P}artition-{F}requency {E}numeration matrix, 2021.
\newblock \href{https://arxiv.org/abs/2102.04191}{arXiv:2102.04191}.

\bibitem{BD2020}
K.~Banerjee and M.~G. Dastidar.
\newblock Hook type tableaux and partition identities.
\newblock {\em Preprint}, 2020.

\bibitem{DS2013}
M.~G. Dastidar and S.~Sen~Gupta.
\newblock Generalization of a few results in integer partitions.
\newblock {\em Notes in Number Theory and Discrete Mathematics}, 19:69--76,
  2013.

\bibitem{Gilbert2015}
R.~A. Gilbert.
\newblock A {F}ine rediscovery.
\newblock {\em Amer. Math. Monthly}, 122(4):322--331, 2015.

\bibitem{Stanley2012}
R.~P. Stanley.
\newblock {\em Enumerative combinatorics. {V}olume 1}, volume~49 of {\em
  Cambridge Studies in Advanced Mathematics}.
\newblock Cambridge University Press, Cambridge, second edition, 2012.

\bibitem{wilf2006}
H.~S. Wilf.
\newblock {\em generatingfunctionology}.
\newblock A K Peters, Ltd., Wellesley, MA, third edition, 2006.

\end{thebibliography}

\end{document}